\documentclass[11pt]{amsart}
\usepackage{graphicx, caption, subcaption, pinlabel}
\usepackage[mathcal]{euscript}
\usepackage{float}
\usepackage{amsmath,amsthm,amssymb,amsfonts,amscd,latexsym,graphicx,textcomp}
\restylefloat{figure}

\newtheorem{theorem}{Theorem}[section]
\newtheorem{lemma}[theorem]{Lemma}
\newtheorem{corollary}[theorem]{Corollary}

\newtheorem{conjecture}[theorem]{Conjecture}

\newtheorem{question}[theorem]{Question}

\theoremstyle{definition}
\newtheorem{definition}[theorem]{Definition} 
\newtheorem{example}[theorem]{Example}

\theoremstyle{remark}
\newtheorem{remark}[theorem]{Remark}

\numberwithin{equation}{section}

\newcommand{\RR}{{\mathbb R}}
\newcommand{\CC}{{\mathbb C}}

\def\del{\partial}

\def\del{\partial}

\begin{document}

\title{Legendrian Products}

\author[P. Lambert-Cole]{Peter Lambert-Cole}

\address{Department of Mathematics, Louisiana State University \newline
\hspace*{.375in} Baton Rouge, LA 70817, USA}

\date{}

\begin{abstract}
This paper introduces two constructions of Legendrian submanifolds of $P \times \RR$, called Legendrian products and spinning, and computes their classical invariants, the Thurston-Bennequin invariant and the Maslov class.  These constructions take two Legendrians $K,L$ and returns a product $K \times L$, they generalize other previous constructions in contact topology, such as frontspinning and hypercube tori, and are equivalent in $\RR^{2n+1}$.  Interestingly, this construction relies upon the explicit embeddings of $K,L$ and not their Legendrian isotopy class.
\end{abstract}

\maketitle
%----------------------------------------------------------------------------------------------%
%----------------------------------------------------------------------------------------------%

\bigskip
\section{Introduction}
\bigskip

The paper introduces a product construction of Legendrian submanifolds in contact manifolds of the form $P \times \RR$.  It gives a multitude of examples of interesting Legendrian submanifolds in a wide array of contact manifolds, not simply affine space and jet spaces, and generalizes some previous constructions.  This construction also seems useful for technical reasons and allows the construction of relative invariants of Legendrian knots and further study of spaces of Legendrian knots.  We also compute the Thurston-Bennequin invariant and Maslov class for these product Legendrians in terms of geometric and topological properties of the factors and obtain an explicit formula when the ambient contact manifold is $\RR^{2n+1}$.  This calculation demonstrates that the Legendrian isotopy class of the product is determined by the particular embedding of a Legendrian factor and not its isotopy class. It shows that the contact geometry of these spaces is extremely rich.

The motivation for this paper is exploring the contact geometry of manifolds $P \times \RR$, where $P$ is an exact symplectic manifold, by constructing many examples of Legendrian submanifolds.  The Legendrian submanifolds of $\RR^3$, Legendrian knots, are well-explored and have found importance in low-dimensional topology.  However, much less is known about higher dimensional ($\geq 5$) contact manifolds and their Legendrian submanifolds.  Several authors have built examples, including the frontspinning construction of Ekholm, Etnyre and Sullivan \cite{high}, a generalization to spheres of arbitrary dimensions by Golovko \cite{Gol}, hypercube tori defined by Baldridge and McCarty \cite{BM} and tori given as the trace of Legendrian isotopies by Ekholm and Kalman \cite{trace}.  In fact, the former three construction should naturally be understood as products in the sense of this paper. 

As smooth knotting is a codimension 2 phenomenom, Legendrian submanifolds in higher dimensions lose some of their relevance to smooth topology of their ambient spaces.  Yet there are indications that the contact geometry of these spaces is rich with Legendrian submanifolds and this paper gives more evidence for this suspicion.  For instance, Knot Contact Homology is an invariant of smooth knots in $\RR^3$, defined by Lenhard Ng, that assigns to each knot a Legendrian torus in $ST^*\RR^3 \simeq T^*S^2 \times \RR$.  Legendrian Contact Homology is a Legendrian invariant of the torus and is therefore a smooth invariant of the underlying knot.  As of this writing, there are currently no known examples of two inequivalent smooth knots with isomorphic knot contact homology.  

The construction of Legendrian products is straightforward.  For a detailed explanation of the terminology, see section \ref{sec:prelim}.  Let $P \times \RR,$ and $Q \times \RR$ be contact manifolds such that $(P,d \lambda), (Q, d \eta)$ are exact symplectic manifolds and with contact forms $\alpha = dz - \lambda$, $\beta = dz - \eta$, respectivly.  Take Legendrian submanifolds $K \in P \times \RR$ and $L \in Q \times \RR$ with Reeb chords $\{a_i\}, \{b_j\}$ and let $\bar{K}, \bar{L}$ denote their Lagrangian projections in $P, Q$.  Then $\bar{K} \times \bar{L}$ is an exact Lagrangian submanifold of $P \times Q$.
\begin{definition}
\label{def:LegProd}
The \textit{Legendrian product} $K \times L$ is the Legendrian submanifold in $P \times Q \times \RR$ given by the lift of $\bar{K} \times \bar{L}$.  It is an immersed Legendrian submanifold and is embedded if the sets of Reeb chord actions $\{Z(a_i)\}, \{Z(b_j)\}$ are disjoint.
\end{definition}
\textbf{Caution}.  This construction is \textit{not} well-defined up to Legendrian isotopy of one of the factors.  If $K(t)$ is a continuous, one-paramter family of embedded Legendrians in $P \times \RR$, it is not true that $K(t) \times L$ gives a Legendrian isotopy through embeddings.  Legendrian isotopies can introduce or annihilate Reeb chords and change Reeb chord actions, which may introduce self-intersections of $K(t) \times L$ for some $t$.  Thus the product construction depends upon the explicit embeddings $K \hookrightarrow P \times \RR$ and $L \hookrightarrow Q \times \RR$ and not simply on the Legendrian isotopy class of $K$ or $L$.

In fact, by varying the Legendrian embedding of a factor within its Legendrian isotopy class, one can obtain infinitely many, non-Legendrian isotopic products.

\begin{theorem}
\label{thrm:infty}
Let $K \in \RR^{2n+1}, L \in \RR^{2m+1}$ be chord generic Legendrians such that $n,m$ have different parity.  Then there exists an infinite family of Legendrians $\{K_i\}$ all Legendrian isotopic to $K$ such that the family of Legendrian products $\{K_i \times L\}$ are pairwise non-Legendrian isotopic.
\end{theorem}

To prove this, we calculate the Thurston-Bennequin invariant of these Legendrian products in terms of the embeddings of $K,L$.  There are two well-known, classical invariants of Legendrian knots: the Thurston-Bennequin number and the rotation number.  These have been generalized to higher dimensions by Tabachnikov and by Ekholm, Etnyre and Sullivan.  
The latter described the Thurston-Bennequin invariant for homologically-trivial Legendrians $L$ as the linking number of $L$ with a pushoff $L'$ of itself along the Reeb vector field
\[tb(L) = lk(L,L')\]
In our setting, when the ambient contact manifold is $P \times \RR$, this can be computed in a manner similar to the writhe of knots. The Maslov class is a cohomology class $\mu \in H^1(L; \mathbb{Z})$ that assigns to each 1-dimensional homology class the Maslov index of a path representing that class.  In dimension 3, the rotation number is $\frac{1}{2} \mu(\gamma)$ where $\gamma$ is a generator of $H_1(S^1; \mathbb{Z})$.
Suppose that $K \in \RR^{2n+1}$ and $L \in \RR^{2m+1}$ are chord generic Legendrians with Reeb chords $\{a_i\}, \{b_j\}$.  Then we can obtain the following formula for the classical invariants of their product.
\begin{theorem}
\label{thrm:TB}
The Thurston-Bennequin number of $K \times L$ is given by:
\[tb(K \times L) = (-1)^{mn} \left( tb(K) \chi(T^*L) + \chi(T^*K) tb(L) + tb(K) tb(L) + \sum_{i,j} \tau(a_i,b_j) \sigma(a_i) \sigma (b_j) \right) \]
where 
\[
  \tau(a_i,b_j) = \left\{
  \begin{array}{l l}
    (-1)^{n} & \quad \text{if $\mathcal{Z}(a_i) < \mathcal{Z}(b_j)$}\\
    (-1)^{m} & \quad \text{if $\mathcal{Z}(a_i) > \mathcal{Z}(b_j)$}\\
  \end{array} \right.
\]
The Maslov class of $K \times L$ is given by
\[\mu_{K \times L} = \mu_k \oplus \mu_L \in H^1(K \times L; \mathbb{Z}) \simeq H^1(K; \mathbb{Z}) \oplus H^1(L; \mathbb{Z})\]
\end{theorem}

The last term in the $tb$ formula requires some elaboration.  Locally, each transverse double point is the intersection of two open $n$-disks in a single point.  Just as computing the sign of a knot crossing requires knowing which strand passes over the other, computing the sign of a Reeb chord requires knowing which disk passes "over" the other.  In the product $K \times L$, after a suitable perturbation, there is a Reeb chord $c_{i,j}$ for each pair $(a_i,b_j)$ of Reeb chords of $K$ and $L$.  However, determining which disk is "over" and which is "under" depends upon the relative lengths of the Reeb chords $a_i, b_j$.

\subsection*{Spinning}
Using this product construction, it is also possible to build examples of Legendrian submanifolds in closed contact manifolds.  Let $M$ be a contact manifold, $L \subset M$ a Legendrian submanifold, $M \hookrightarrow M'$ a contact embedding of codimension $2m$ with trivial conformal symplectic normal bundle and $K \subset \RR^{2m+1}$ a Legendrian submanifold.  Then we can identify a neighborhood of $L$ with a neighborhood of the 0-section in $J^1(L) \times \RR^{2m}$ and find $L \times K$ as a Legendrian submanifold of $M'$, which we denote by $K \times_M L$ and refer to as the \textit{spinning} of $K$ by $L$.  It is clear from the construction that this is invariant under Legendrian isotopies of $L$ in $M$, contact isotopies of $M$ in $M'$ and compactly supported isotopies of $K$.

In \cite{high}, Etnyre, Ekholm and Sullivan defined a construction, called frontspinning, which takes a Legendrian $L \subset \RR^{2n-1}$ and produces a Legendrian $\Sigma L \subset \RR^{2n+1}$ of topological type $L \times S^1$.  In fact, these should more naturally be seen as the product of $L$ with a standard Legendrian unknot with $tb = -1, r = 0$.  Consider $S^{2k-1}$ as the unit sphere in $\RR^{2k}$ and take the obvious contact embedding $S^{2k-1} \hookrightarrow \RR^{2n+1} \simeq \RR^{2k} \times \RR^{2n-2k+1}$.  We prove the following theorem and corollary.
\begin{theorem}
\label{thrm:spin}
Let $L$ be a Legendrian submanifold of $S^{2k-1}$ (equiv $\RR^{2k-1}$) and $K$ a Legendrian submanifold of $\RR^{2n-2k+1}$, chosen so that all Reeb chord actions of $K$ are much less than all Reeb chord actions of $L$.  Then $K \times_{S^{2k-1}} L$ and $K \times L$ are Legendrian isotopic in $\RR^{2n+1}$
\end{theorem}
Let $W$ denote the Whitney embedding of the standard Legendrian unknot of $S^3$ with exactly one Reeb chord of length 1.
\begin{corollary}
\label{cor:front}
Let $L \subset \RR^{2n-1}$ be Legendrian all of whose Reeb chords have action $\mathcal{Z}(c) \ll 1$.  Then $\Sigma L$ and $W \times L$ are Legendrian isotopic in $\RR^{2n+1}$.
\end{corollary}

Golovko has recently \cite{Gol} extended the frontspinning construction to produce Legendrians in $\RR^{2n+2k-1}$ of topological type $L \times S^k$ and this can be interpreted as products with Whitney spheres of arbitrary dimension.

\subsection*{Legendrian contact homology}
The Thurston-Bennequin invariant is sometimes blind, as in \cite{high} it is shown to reduce to the Euler characteristic when the ambient manifold is $\RR^{2n+1}$ for even $n$.  So the Thurston-Bennequin invariant is not useful for distinguishing products when the dimensions of the factors have the same parity. The results here suggest that a more powerful invariant, such as Legendrian contact homology, will reveal many interesting examples of Legendrian submanifolds in higher dimensional contact manifolds. 

Legendrian contact homology is an invariant of Legendrian submanifolds that fits into the general framework of Symplectic Field Theory.  It has been constructed by Ekholm, Etnyre and Sullivan when the ambient contact manifold is some $P \times \RR$ \cite{LCH}, \cite{Jets}.  To each Legendrian $L$, it associates a diferential graded algebra, which is invariant up to an appropriate algebraic equivalence.  The algebra is the unital tensor algebra over the group ring $\mathbb{Z}[H_1(L)]$ generated by the Reeb chords of $L$.  The differential counts rigid, punctured holomorphic maps $u: (D^2,\del D^2) \rightarrow (P,\overline{L})$ with a single positive puncture and arbitrarily many negative punctures.  In dimension 3, the Legendrian contact homology of Legendrian knots is known as the Chekanov-Eliashberg DGA and by the Riemann mapping theorem, can be computed combinatorially from either the front or Lagrangian projection of a knot.  

The Legendrian contact homology of the product $L \times K$ can be computed in terms of geometric and topological data of the factors (see \cite{product}).  Holomorphic disks on the product $L \times K$ are determined by and can be constructed from holomorphic disks and gradient flow lines on the factors.  However, an interesting complication arises from the dependence on relative Reeb chord lengths that first appeared in the Thurston-Bennequin calculation.  Specifically, the LCH of the product potentially includes information determined by holomorphic disks with multiple positive punctures.  This information is ignored in the standard version of Legendrian contact homology.  This suggests that a full understanding of the LCH of products is more interesting than one might expect from the construction and from results of Ekholm, Etnyre and Sabloff on the linearized contact homology of frontspun Legendrians $\Sigma L$ \cite{Duality}.

Given that the LCH of Legendrian knots can be computed combinatorially and that Ekholm has established a correspondence between rigid holomorphic disks on $L \subset J^1(\RR)$ and rigid gradient flow trees on $L$, it is reasonable to conjecture the following:
\begin{conjecture}
The Legendrian contact homology of the products of Legendrian knots can be computed combinatorially.
\end{conjecture}

Finally, for $L \subset \RR^{2n+1}$ and $K \subset \RR^{2m+1}$ suitably small, theorem \ref{thrm:spin} shows that the product $L \times K$ is well-defined up to Legendrian isotopy of the factors.  Thus, the Legendrian contact homology of the product $L \times K$ is a Legendrian invariant of both of the factors.  Yet since it depends upon geometric data not included in the LCH of the factors, this invariant may provide more information about the Legendrian isotopy class of $L$ and $K$ than the LCH of the factors individually.
\begin{question}
(Relative Legendrian invariants) Is LCH($L \times K$) a stronger invariant of $L$ than $LCH(L)$ for some $K$?  
\end{question}
In particular, one may fix $L$ and vary $K$, using the LCH of the product to probe the contact geometry of $L$.

\subsection*{Organization}  The outline of this paper is as follows.  Section \ref{sec:prelim} reviews some preliminary content and definitions about contact geometry, Legendrians, Reeb chords in the relevant setting of manifolds of the form $P \times \RR$.  Section \ref{sec:TB} contains the proofs of theorems \ref{thrm:TB} and \ref{thrm:infty}.  This first requires constructing a suitable perturbation of the conormal lift as it is not chord generic.  Section \ref{sec:spin} defines spinning and proves theorem \ref{thrm:spin} and \ref{cor:front}.  Section \ref{sec:exam} discusses examples, including products of Whitney spheres and products of Legendrian knots.  

\section*{Acknowledgements}
I would like to thank my advisor, Scott Baldridge, for advice and encouragement, as well as John Etnrye, Forrest Gordon, Kate Kearney, Ben McCarty and Shea Vela-Vick for helpful conversations.

\section{Preliminaries}
\label{sec:prelim}

Let $P$ be an exact symplectic manifold of dimension $2n$, meaning the symplectic form $\omega$ can be expressed as $\omega = d \lambda$ some primitive $\lambda$.  A submanifold $L \subset P$ is \textit{isotropic} if the restriction $\omega|_L$ is identically 0 and \textit{Lagrangian} if it is isotropic and has dimension $n$.  A Lagrangian submanifold is \textit{exact} if the restriction $\lambda_L$ is exact.  For any (immersed) Lagrangian $L$, there is an oriented diffeomorphism between the normal bundle $\nu(L)$ of $L$ in $P$ and the cotangent bundle $T^*L$.  Moreover, there is a symplectomorphism between some neighborhood of $L$ in $P$ and a neighborhood of the 0-section in $T^*L$.

The 1-form $\alpha = dz - \lambda$ is a contact form on the product manifold $P \times \RR$ ($z$ is the coordinate on the second factor) as $\alpha \wedge (d \lambda)^n = \omega^n \wedge dz \neq 0$.  This contact form induces a contact structure $\xi = \text{ker} (\alpha)$, which is a maximally nonintegrable hyperplane field on $P \times \RR$, and a \textit{Reeb vector field} $R_{\alpha} = \partial_z$.  The contact manifold $(P \times \RR, \xi)$ is called the \textit{contactization} of $P$.  A submanifold $L$ of some contact manifold $(M,\xi)$ is \textit{isotropic} if it is everywhere tangent to the hyperplane field $\xi$ and is  \textit{Legendrian} if it is isotropic and has dimension $n$.  A continuous, one-parameter family of Legendrian submanifolds is a \textit{Legendrian isotopy}.  Let $\Pi_P: P \times \RR \rightarrow P$ be the projection map onto the first factor, which is called the \textit{Lagrangian projection}.  Throughout this paper, we distinguish points and sets $x, U \subset P \times \RR$ from their images under $\Pi_P$ through bar notation i.e $\Pi_P(x) = \bar{x}, \Pi_P(U) = \bar{U}$.

The projection $\overline{L}$ of a Legendrian submanifold is exact Lagrangian, since for $Z:L \rightarrow \RR$, the restriction of the projection $\Pi: P \times \RR \rightarrow \RR$ to $L$, the Legendrian condition implies that $\lambda_{\overline{L}} = d Z$.  Furthermore, given an exact Lagrangian submanifold $\overline{L}$ in $P$, there exists a \textit{lift} $L \subset P \times \RR$ of $\overline{L}$ such that $L$ is an immersed Legendrian submanifold.  This can be chosen uniquely up to a translation in the $z$-direction.

A \textit{Reeb chord} $c$ for some Legendrian submanifold $L$ is an integral curve of the Reeb vector field that begins and ends on $L$.  Since the Reeb vector field flows in the $z$-direction, the endpoints $c^+, c^-$ of the chord project to the same point $\bar{c}$ and every multiple point of the projection $\bar{L}$ lifts to at least one Reeb chord. A Legendrian submanifold is \textit{chord generic} if its Lagrangian projection has a finite number of transverse double points.  Double points of $\bar{L}$ correspond to unique Reeb chords since $\RR$ is not compact and this implies that $L$ has a finite number of Reeb chords.  The \textit{action} $\mathcal{Z}(c)$ of a Reeb chord is its length, which is equal to the difference of $z$-coordinates $Z(c^+) - Z(c^-)$.  Choose neighborhoods $U_+, U_-$ around $c^+, c^-$ and call these the \textit{upper sheet} and \textit{lower sheet}.  Each Reeb chord has a sign $\sigma(c)$ defined as follows.  Let $V_+ := (\Pi_{P})_*(T_{c^+}L)$ and $V_- := (\Pi_{P})_*(T_{c^-}L)$.  Since $c$ corresponds to a transverse double point in the Lagrangian projection, $V_+ \oplus V_-$ span $T_{\bar{c}}P$.  If the orientation of $V_+ \oplus V_-$ agrees with the orientation of $P$, then $\sigma(c) = 1$; otherwise $\sigma(c) = -1$.

Let $(M, \alpha)$ be an oriented contact manifold with contact structure $\xi = \text{ker}(\alpha)$ and recall that the 2-form $d \alpha$ restricts to a symplectic form on $\xi$.  If a submanifold $L$ is isotropic, then $TL \subset \xi|_L$ and define $TL^{\perp}$ to be the symplectic subbundle of $\xi|_L$ whose fibers are the symplectic orthogonal complements to the fibers of $TL$.  The \textit{conformal symplectic normal bundle} of $L$ in $M$ is the quotient bundle
\[CSN_M(L) = TL^{\perp}/TL\]
If dim$M = 2n+1$ and dim $L = m$ then $CSN_M(L)$ has rank $2n-2m$.  If we choose an almost complex structure $J$ on $\xi$ compatible with $d \alpha$, then the normal bundle of $L$ in $M$ splits as
\[NL = <R_{\alpha}> \oplus J(TL) \oplus CSN_M(L)\]
The contact form restricts to a contact form on the fibers of $< R_{\alpha} > \oplus CSN_M(L)$ and so it is a contact subbundle of $NL$.  Furthermore, there exists a contactomorphism between suitable neighborhoods of $L \subset M$ and the 0-section of $J^1(L) \oplus CSN_M(L)$.
Similarly, given a contact embedding $(M, \xi) \hookrightarrow (M', \xi')$, the \textit{conformal symplectic normal} bundle $CSN_{M'}(M)$ is the symplectic subbundle $(\xi)^{\perp} \subset \xi'|_M$ given by taking the symplectic orthogonal complement to $\xi$ in $\xi'|_M$.  This bundle can be identified with the normal bundle $NM$ of $M$ in $M'$.  There also exists a contactomorphism between suitable neighborhoods of $M \subset M'$ and the 0-section of $CSN_{M'}(M)$.

\subsection*{Maslov Class}
Let $\Lambda_n$ be the Grassman manifold of Lagrangian subspaces in the standard affine symplectic space $(\RR^{2n},\omega)$.  Fix some Lagrangian subspace $\Lambda \in \Lambda_n$ and let $\Sigma_k \subset \Lambda_n$ be the set of all Lagrangian planes in $\RR^{2n}$ that intersect $\Lambda$ along a subspace of dimension $k$.  Then the \textit{Maslov cycle} is the algebraic subvariety
\[\Sigma = \overline{\Sigma_1} = \Sigma_1 \cup \dots \cup \Sigma_n\]
which has codimension 1 in $\Lambda_n$.  For a path $\Gamma: [0,1] \rightarrow \Lambda_n$, we can define an intersection number of $\Gamma$ and $\Sigma$ as follows.  Fix a Lagrangian complement $W$ to $\Lambda$ and suppose that $\Gamma(t')$ intersects $\Sigma$.  For $t$ near $t'$, there exists a family $w(t) \in W$ of vectors such that for all $v \in \Gamma(t') \cap \Sigma$ the vector $v + w(t) \in \Gamma(t)$.  Then there is a quadratic form $Q = \frac{d}{dt}|_{t'} \omega(v, w(t))$ on $\Gamma(t') \cap \Sigma$ and the signature of this quadratic form is the intersection number of $\Gamma$ at $t'$.  

If $\Gamma$ is a loop, then the \textit{Maslov index} $\mu(\Gamma)$ is the total intersection number of $\Gamma$ with $\Sigma$.  The map $\mu$ defines an isomorphism $H_1(\Lambda_n) \simeq \pi_1(\Lambda_n) \simeq \mathbb{Z}$.

Let $\overline{L} \subset \RR^{2n}$ be an immersed Lagrangian.  A global trivialization of $T\RR^{2n}$ induces a map $f: L \rightarrow \Lambda_n$, where each point is sent to its Lagrangian tangent plane.  The \textit{Maslov class} is the pullback of a generator $m$ of $H^1(\Lambda_n;\mathbb{Z})$
\[\mu_L := f^*(m)\]

\section{Classical Invariants}
\label{sec:TB}

The goal of this section is to prove theorems \ref{thrm:TB} and \ref{thrm:infty}.

One problem with this product is that while each factor is chord generic, the product is not.  In fact, each Reeb chord is part of some family of Reeb chords given by either $a_i \times L$ or $K \times b_j$.  In order to make $K \times L$ chord generic, we must perturb it slightly.  Let $f,g$ be $C^1$-small Morse functions on $K,L$ whose critical points are away from the endpoints of the Reeb chords and such that the endpoints lie in different level sets (i.e. $f(a_i^+) \neq f(a_i^-), g(b_j^+) \neq g(b_j^-)$).  Thus, there exist neighborhoods $U_i^+,U_i^-$ of $a_i^+,a_i^-$ such that $f(U_i^+)$ and $f(U_i^-)$ are disjoint and similarly there exist such neighborhoods $W_j^+,W_j^-$ of each $b_j^+,b_j^-$.  Denote the critical points of $f$ by $m^1_k$ and the critical points of $g$ by $m^2_l$.  We can identify a small neighborhood of $K \times L$ with a neighborhood of the 0-section in its 1-jet space $J^1(K \times L)$ and perturb it by a Legendrian isotopy to the graph of $fg$ in the 1-jet space $J^1(K \times L)$.  

\begin{lemma}
The perturbed Legendrian is chord generic and has the following Reeb chords:
\begin{itemize}
\item \textbf{Reeb/Morse}: one for each pair $(a_i, m_l^2)$ of Reeb chord for $K$, Morse critical point of $L$, denoted $a_i \otimes m_l^2$
\item \textbf{Morse/Reeb}: one for each pair $(m_k^1, b_j)$ of Morse critical point of $K$, Reeb chord for $L$, denoted $m_k^1 \otimes b_j$
\item \textbf{Reeb/Reeb}: two for each pair $(a_i, b_j)$ of Reeb chord for $K$, Reeb chord for $L$, denoted $c_{i,j}$ and $d_{i,j}$
\end{itemize}
\end{lemma}
We will refer to these as A-chords, B-chords, C-chords and D-chords, respectively.
\begin{proof}
In the Lagrangian projection, a neighborhood of $\overline{K \times L}$ is symplectomorphic to a neighborhood $\eta(0)$ of the 0-section of the cotangent bundle.  Moreover, we can assume that the map $v(0) \rightarrow P \times Q$ is injective away from $T^*(U_i^{\pm} \times W_j^{\pm})|_{\eta}$.  The perturbation pushes $K \times L$ off to the graph of $d(fg) = g df + f dg$.  Let $x,y$ denote points in $K,L$ and $\bar{x}, \bar{y}$ denote those points projected to $P, Q$.  Then the perturbation maps $(\bar{x}, \bar{y})$ in $P \times Q$ to $(\bar{x} + g(y) df(x), \bar{y} + f(x) dg(y))$. 

Now, suppose that $(\bar{x} + g(y) df(x), \bar{y} + f(x) dg(y)) = (\bar{x'} + g(y') df(x'), \bar{y'} + f(x') dg(y'))$ for some $x,x' \in K, y,y' \in L$.  If $x = x'$, then either $x$ is a Morse critical point or $y,y'$ lie in the same level set. In the first case, we get intersection points coming from the intersection points of $f(x) dg$, the pushoff, whose intersection points are in 1-1 correspondance with the intersection points of $L$.  This gives the A-chords.  The second is impossible, since if $y,y'$ are distinct then we must have (up to relabeling) that $y \in W_j^{+}$ and $y' \in W_j^-$ and so $y,y'$ cannot lie in the same level set.  Repeating this for $y = y'$ will yield the B-chords.

Finally, consider when both pairs $x,x'$ and $y,y'$ are distinct.  Prior to the perturbation, there was a unique transverse intersection point of $\overline{U_i^+} \times \overline{W_j^+}$ and $\overline{U_i^-} \times \overline{W_j^-}$.  Similarly, there was a unique transverse intersection point of $\overline{U_i^+} \times \overline{W_j^-}$ and $\overline{U_i^-} \times \overline{W_j^+}$.  Since the perturbation is $C^1$-small, we can assume that after the perturbation, there remains a unique intersection point in each case.  Thus, the intersection point $(\bar{x},\bar{y}) = (\bar{x'},\bar{y'})$ under consideration must be one of these two; the first case we label $d_{i,j}$ and the second we label $c_{i,j}$.
\end{proof}

We can compute the signs of the intersection points as well:

\begin{lemma}
\label{lem:sign}
The Reeb chords of the perturbation have the following actions
\begin{itemize}
\item $\mathcal{Z}(a_i \otimes m_l^2) \approx \mathcal{Z}(a_i)$
\item $\mathcal{Z}(m_k^1 \otimes b_j) \approx \mathcal{Z}(b_j)$
\item $\mathcal{Z}(c_{i,j}) \approx |\mathcal{Z}(a_i) - \mathcal{Z}(b_j)|$
\item $\mathcal{Z}(d_{i,j}) \approx \mathcal{Z}(a_i) + \mathcal{Z}(b_j)$
\end{itemize}
and signs
\begin{itemize}
\item $\sigma(a_i \otimes m_l^1) = (-1)^{mn} \sigma(a_i) \sigma(m_l^2)$
\item $\sigma(m_k^1 \otimes b_j) = (-1)^{mn} \sigma(m_k^1) \sigma(b_j)$
\item $\sigma(c_{i,j}) = (-1)^{mn} \sigma(a_i) \sigma(b_j) \tau(i,j)$
\item $\sigma(d_{i,j}) = (-1)^{mn} \sigma(a_i) \sigma(b_j)$
\end{itemize}
where
\[
  \tau(i,j) = \left\{
  \begin{array}{l l}
    (-1)^{n} & \quad \text{if $\mathcal{Z}(a_i) < \mathcal{Z}(b_j)$}\\
    (-1)^{m} & \quad \text{if $\mathcal{Z}(a_i) > \mathcal{Z}(b_j)$}\\
  \end{array} \right.
\]
\end{lemma}
\begin{proof}
As above, let $V_+ := (\Pi_P)_*(T_{a^+}K)$ and $V_- := (\Pi_P)_*(T_{a^-}K)$, which we think of as the tangent planes to the upper and lower sheets $\bar{U_+}, \bar{U_-}$ in $P$ at $\bar{a}$, and define $W_+,W_-$ similarly for $L$.  Then for the Reeb/Morse chords, the tangent planes of the upper and lower sheets at $\overline{a_i \otimes m^1_l}$ are given by $V_+ \oplus W_+$ and $V_- \oplus W_-$.  The sign of the Reeb chord is given by the orientation of $V_+ \oplus W_+ \oplus V_- \oplus W_-$, which is $(-1)^{mn}$ times the orientation given by $V_1+ \oplus V_- \oplus W_+ \oplus W_-$, whose sign is given by $\sigma(a_i) \sigma( m_k^2)$.  The situation is similar for the Morse/Reeb chords and each $d_{i,j}$.  

However, for the $c_{i,j}$, the calculation is different.  If $\mathcal{Z}(a_i) < \mathcal{Z}(b_j)$, the tangent plane to the upper sheet at $\bar{c_i,j}$ is given by $V_+ \oplus W_-$ and to the lower sheet by $V_- \oplus W_+$.  Thus, the orientation on $V_+ \oplus W_- \oplus V_- \oplus W_+$ is $(-1)^{mn}(-1)^{m^2}$ times that of $V_+ \oplus V_- \oplus W_+ \oplus W_-$, which is given by $\sigma(a_i) \sigma(b_j)$.  But if $\mathcal{Z}(a_i) > \mathcal{Z}(b_j)$, then the sign of the Reeb chord is given by the orientation of $V_- \oplus W_+ \oplus V_+ \oplus W_-$, which differs from $\sigma(a_i) \sigma(b_j)$ by $(-1)^{mn}(-1)^{n^2}$.
\end{proof}

We can now prove Theorem \ref{thrm:TB}.  Suppose that $\bar{L}$, the projection of $L$ to $P$, has a finite number of transverse double points.  Then these double points are in one-to-one correspondance with the Reeb chords and $L$ is called chord generic.  Each chord $c$ can then be assigned a sign $\sigma(c) = \pm 1$ and in \cite{high}, it is shown that the Thurston-Bennequin invariant is given by
\[tb(L) = \sum_c{\sigma(c)}\]
Proof of Theorem \ref{thrm:TB}
\begin{proof}
The Thurston-Bennequin calculation follows directly from lemma \ref{lem:sign} by summing over all indices $i,j,k,l$.

In order to calclate the Maslov class, we can make the necessary choices so that each condition splits.  Specifically, take $\Lambda' \in \Lambda_n$ and $\Lambda'' \in \Lambda_m$ and associated Maslov cycles $\Sigma(\Lambda'), \Sigma(\Lambda'')$.  Then $\Lambda = \Lambda' \oplus \Lambda''$ is Lagrangian in $\RR^{2n+2m}$ and so defines a Maslov cycle $\Sigma(\Lambda)$ that splits as
\[\Sigma(\Lambda)_k = \sum_{i + j = k} \Sigma(\Lambda')_i \oplus \Sigma(\Lambda'')_j\]
Moreover, we can choose Lagrangian complements such that $W = W' \oplus W''$.

Take a path $\Gamma \in \Lambda_{n+m}$ and its projections $\Gamma' \in \Lambda_n, \Gamma'' \in \Lambda_m$.  At each intersection point $\Gamma(t') \subset \Sigma(\Lambda')$, the signature of the associated quadratic form is the sum of the signatures of the associated quadratic forms for the intersections of $\Gamma'(t'), \Gamma''(t')$ and $\Sigma(\Lambda'), \Sigma(\Lambda'')$.

Thus, the Maslov index splits as $\mu(\Gamma) = \mu(\Gamma') + \mu(\Gamma'')$.  Therefore, it follows that the Maslov class $\mu_{L \times K}$ splits as well.
\end{proof}

\begin{remark}
Notice that this formula is consistent with the result in \cite{high} that for $L$ of even dimension, the Thurston-Bennequin number is $-\frac{1}{2} \chi (\nu)$, where $\nu$ is the oriented normal bundle to $L$ in the Lagrangian projection.  Suppose that the dimensions of $K$ and $L$ have the same parity.  Then the dimension of their product is even and so this result applies.  If $n,m$ are odd, then the Euler characteristic of both manifolds vanish by Poincare duality, the Euler characteristic of their product vanishes, and $\tau_{i,j}$ is always negative.  Thus, we get that:
\[tb(L_1 \times L_2)  = -( 0 + 0  + tb(K) tb(L) - tb(K) tb(L)) = 0 = \chi (K \times L)\]
If $n,m$ are even, then $\tau_{i,j}$ is always 1 and we can use the immersed version of the Lagrangian Neighborhood Theorem to identify $T^*K$ with $\nu_K$ and $T^*L$ with $\nu_L$ and obtain
\begin{align*}
tb(K \times L) &= \left( -\frac{1}{2} \right) \chi(\nu_K)  \chi (\nu_L) + \chi(\nu_K)  \left(-\frac{1}{2}\right) \chi (\nu_L) \\
	&+ \left(-\frac{1}{2}\right) \chi(\nu_K)  \left(-\frac{1}{2}\right) \chi (\nu_L) + \left(-\frac{1}{2}\right) \chi(\nu_K)  \left(-\frac{1}{2}\right) \chi (\nu_L) \\
 &= - \frac{1}{2}  \chi(\nu_K)  \chi (\nu_L) \\
 &= -\frac{1}{2}\chi (\nu_{K \times L})
\end{align*}
since it is clear from the construction that the normal bundle of the product is the product of the normal bundles.
\end{remark}

Proof of theorem \ref{thrm:infty}
\begin{proof}
Choose some Darboux ball of radius $3\epsilon$ and isotope $K$ so that its intersection with the Darboux ball is two disjoint disks given by two parallel Lagrangian planes of distance $\epsilon$ apart.  By a Hamiltonian isotopy supported in the Darboux ball, we can add two canceling transverse double points corresponding to two Reeb chords $b,a$, labeled so that $\mathcal{Z}(a) > \mathcal{Z}(b)$.  Now, by scaling either $K$ or $L$ and a perturbation, we can assume that there is exactly one Reeb chord $e$ of $L$ such that $\mathcal{Z}(a) > \mathcal{Z}(e) > \mathcal{Z}(b)$.  For all other Reeb chords $e'$ of $L$, the terms $\tau(a,e') \sigma(a) \sigma(e')$ and $\tau(b,e') \sigma(b) \sigma (e')$ cancel.  However, $\tau(a,e) \sigma(a) \sigma(e)$ and $\tau(b,e) \sigma(b) \sigma (e)$ have the same sign since the pairs $\tau(a,e),\tau(b,e)$ and $\sigma(a), \sigma(b)$ are distinct.  Moreover, we can add aribtrarily many pairs of Reeb chord pairs $\{(a_i,b_i)\}$ so that $\mathcal{Z}(a_i) = \mathcal{Z}(a_j)$ and $\mathcal{Z}(b_i) = \mathcal{Z}(b_j)$ for all $i,j$.  Thus, we can add $2n * (\tau(a,e) \sigma(a) \sigma(e))$ to the Thurston-Bennequin invariant for arbitrary nonnegative integer $n$.
\end{proof}

\section{Spinning}
\label{sec:spin}
The goal of this section is to prove theorem \ref{thrm:spin} and corollary \ref{cor:front}.  We first define frontspinning as in \cite{high}, then give an alternate, invariant description of this construction called spinning before proving that spinning and the Legendrian product are equivalent in $\RR^{2n+1}$.

Frontspinning is a construction defined by Ekholm, Etnyre and Sullivan that takes a Legendrian $L \subset \RR^{2n+1}$ and produces a Legendrian $\Sigma L \subset \RR^{2n+3}$ of topological type $L \times S^1$.  Suppose that a Legendrian $L \subset \RR^{2n+1}$ is given by the embedding $f: L \rightarrow \RR^{2n+1}$ and parametrized such that
	\[f(L) = (x_1(L), \dots, x_n(L), y_1 (L), \dots, y_n(L), z(L))\]
Let $\Pi_F: \RR^{2n+1} \rightarrow \RR^{n+1}$ be the projection onto the $x$ and $z$ coordinates.  The \textit{front projection} of $L$ is the subsvariety
\[ \Pi_F(L) = (x_1(L), \dots, x_n(L), z(L))\]
If $S$ is a subvariety of $\RR^{n+1}$ such that $\frac{\partial}{\partial z} \notin T_x S$ for any $x \in S$, then $S$ lifts to an isotropic subvariety of $\RR^{2n+1}$ as the $y$-coordinates can be chosen at each point to satisfy the contact condition.  The \textit{frontspinning} of $L$ is the lift to $\RR^{2n+3}$ of the following subvariety of $\RR^{n+2}$:
\[S = (\cos \theta x_1(L), \sin \theta x_1 (L), x_2(L), \dots, x_n(L), z(L))\]
for $\theta \in [0,2 \pi]$.  This can be thought of as spinning the front projection of $L$ around the plane $x_0 = x_1 = 0$ in $\RR^{n+2}$ and the intersections of $S$ with the planes $(\cos \theta x_1, \sin \theta x_1, x_2, \dots, x_n, z)$ give a family of front projections of $L$ as $\theta$ varies.

We now give an alternate interpretation of this construction that generalizes to arbitrary contact manifolds.  Choose some Legendrian $K \subset M$, a contact embedding $M \hookrightarrow M'$ of codimension $2m$ with trivial $CSN_{M'}(M)$ and a Legendrian $L \subset \RR^{2m+1}$.  By scaling and translation, we can assume that $L$ lies in a suitably small neighborhood of the origin.  It follows that the conformal symplectic normal bundle of $K$ in $M'$ splits as
\[CSN_{M'}(K) = CSN_M(K) \oplus CSN_{M'}(M)\]
and by assumption, since $K$ is Legendrian, this bundle is trivial of rank $2m$.  
\begin{definition}
The \textit{spinning} of $L$ by $K$, denoted $K \times_M L$, is the Legendrian submanifold of $M'$ obtained as the image of $K \times L$ under the contactomorphism that identifies neighborhoods of $K$ in $J^1(K) \times \RR^{2m}$ and $M'$.
\end{definition}
It is clear from the construction that ambient contact isotopies of $K, M$ and $L$ extend to Legendrian isotopies of $K \times_M L$, provided that $L$ is contained in a suitably small neighborhood of 0.

Consider $(S^3, \xi_{std})$ as the unit sphere in $\RR^4$.  Then for the embedding $(S^3, \xi_{std}) \hookrightarrow \RR^{2n+1} \simeq \RR^4 \times \RR^{2n-3}$, the conformal symplectic normal bundle $CSN_{\RR^{2n+1}}(S^3)$ is trivial.  Let $W$ be the submanifold
\[W = (\cos \theta, 0 , \sin \theta, 0) \text{,  } \theta \in S^1\]
It is easy to verify that W is Legendrian and isotopic to the standard Legendrian unknot with $tb = -1, r = 0$.

\begin{lemma}
\label{lemma:sigmaL}
The Legendrian submanifolds $\Sigma L$ and $W \times_{S^3} L$ are identical.
\end{lemma}

\begin{proof}
Let $J$ be the standard complex structure on $\CC^2$.  Then we can trivialize $TW, J(TW)$ as
\[T_{\theta}W = <-\sin \theta \partial_{x_1} + \cos \theta \partial_{x_2}>\]
\[J(T_{\theta}W) = < -\sin \theta \partial_{y_1} + \cos \theta \partial_{y_2}>\] 
and trivialize $CSN_{\RR^{2n+1}}(W)$ as
\[CSN_{\RR^{2n+1}}(W) = <\cos \theta \partial_{x_1} + \sin \theta \partial_{x_2}, \del_{x_3}, \del_{y_3}, \dots, \del_{x_n}, \del_{y_n}, \del_z >\]
The Reeb vector field in $S^3$ along $W$ is given by $R = \cos \theta \del_{y_1} + \sin \theta \del _{y_2}$.  It follows from above that the frontspun $\Sigma L$ is obtained by restricting to $L$ in each fiber of the bundle $<R> \oplus CSN_{\RR^{2n+1}}(W)$
\end{proof}

We can now prove theorem \ref{thrm:spin}
\begin{proof}
Choose a Darboux ball around the point $(1, 0, \dots, 0) \in \RR^{2n+1}$, given by some map $f: B^{2n+1} \rightarrow \RR^{2n+1}$ of the unit ball, that restricts to a Darboux ball $f':B^{2k-1} \rightarrow S^{2k-1}$ on $S^{2k-1}$ as well.  Thus, $f(K \times L)$ is exactly $f'(K) \times_{S^{2k-1}} L$.  Choose some Legendrian isotopy $K(t)$ in $S^{2k-1}$ so that $K(0) = K$ and $K(1) = f'(K)$ in this Darboux ball on $S^{2k-1}$.  The isotopy extends to some suitable neighborhood of $K$ and isotope $L$ so that $K \times_{S^{2k-1}} L$ lies in this neighborhood.  The Legendrian isotopy of $K$ thus extends to a Legendrian isotopy from $K \times_{S^{2k-1}} L$ to $f(K \times L)$.

Recall that the contact disk theorem states that for any two contact embeddings $g,h: B^{2n+1} \rightarrow M^{2n+1}$ of the unit ball into a contact manifold, there is a contact isotopy $i: M \rightarrow M$ such that $i  \circ h = g$.  Thus, there exists some isotopy $i: \RR^{2n+1} \rightarrow \RR^{2n+1}$ such that $i \circ f = id$ and this isotopy sends $f(K \times L)$ to $K \times L$.
Furthermore, since $L$ must live in a suitably small neighborhood of the origin, whose diameter must be less than the length of the Reeb chords of $K$, it follows that all its Reeb chord actions must be less than this diameter.
\end{proof}

Corollary \ref{cor:front} now follows directly from lemma \ref{lemma:sigmaL} and theorem \ref{thrm:spin}.

\section{Examples}
\label{sec:exam}
\subsection{Whitney spheres}
A Whitney sphere $W_c^n$ is the Legendrian sphere $S^n \in \RR^{2n+1}$ given as the lift of the following exact Lagrangian immersion in $\CC^n \simeq \RR^{2n}$: 
\[w: S^n = \{(x,y) \in \RR^n \times \RR: |x|^2 + y^2 = 1\} \mapsto c(1 + i y)x\]
where $c$ is some positive real constant.  This immersion has exactly 1 transverse double point at $w(0,1) = w(0,-1)$ and so the Legendrian $W_c^n$ has exactly one Reeb chord, of length determined by $c$.

\begin{example}
Consider the product $W_a^1 \times W_b^2$ of a single 1-sphere and a single 2-sphere.  Since the parities of the dimensions are different, the dimension of the product is odd and the Thurston-Bennequin invariant will be useful.  We have 
\begin{align*}
tb(W_a^1) &= -1 \hspace{.3 in} \chi(T^*S^1) = 0 \\
tb(W_b^2) &= 1 \hspace{.3 in} \chi(T^*S^2) = -2 \\
\end{align*}
Denote the Reeb chords $a$ and $b$.  The $\tau$ factor for this pair is
\[\tau(a,b) = \left\{ \begin{array}{l l}
    (-1)^1 = -1 & \quad \text{if $a < b$}\\
    (-1)^2 = 1 & \quad \text{if $a > b$}
  \end{array} \right. \]
Thus
\[tb(W_a^1 \times W_b^2) = \left\{ \begin{array}{l l}
    2 & \quad \text{if $a < b$}\\
    0 & \quad \text{if $a > b$}
  \end{array} \right. \]
Now, consider the product $W_a^1 \times W_b^4$ of a single 1-sphere and a single 4-sphere.  We have 
\begin{align*}
tb(W_a^1) &= -1 \hspace{.3 in} \chi(T^*S^1) = 0 \\
tb(W_b^4) &= 1 \hspace{.3 in} \chi(T^*S^4) = 2 
\end{align*}
The $\tau$ factor for this pair is
\[\tau(a,b) = \left\{ \begin{array}{l l}
    (-1)^1 = -1 & \quad \text{if $a < b$}\\
    (-1)^4 = 1 & \quad \text{if $a > b$}
  \end{array} \right.\]
Thus
\[tb(W_a^1 \times W_b^4) = \left\{ \begin{array}{l l}
    -2 & \quad \text{if $a < b$}\\
    0 & \quad \text{if $a > b$}
  \end{array} \right. \]
\end{example}
Note that for both examples, the case $a > b$ was already calculated in \cite{high}.
\begin{example}
Let $W_a^1, W_b^1, W_e^1$ be three Whitney unknots.  To calculate the Thurston-Bennequin invariant of $W_a^1 \times W_b^1 \times W_e^1$ apply theorem \ref{thrm:TB} twice.  The torus $W_a^1 \times W_b^1$ has $tb = 0$ and $\chi(T^*T^2) = 0 $ and, after perturbing, has six Reeb chords with actions
\[\mathcal{Z}(a \otimes m_i) \approx \mathcal{Z}(a)\]
\[\mathcal{Z}(m_i \otimes b) \approx \mathcal{Z}(b)\]
\[\mathcal{Z}(c) \approx |\mathcal{Z}(a) - \mathcal{Z}(b)|\]
\[\mathcal{Z}(d) \approx \mathcal{Z}(a) + \mathcal{Z}(b)\]
and signs
\[\sigma(a \otimes m_i) = (-1)\sigma(a)(-1)^i \hspace{.5in} \sigma(m_i \otimes b) = (-1)\sigma(b)(-1)^i\]
\[\sigma(c) = (-1)^2 \sigma(a)\sigma(b) \hspace{.5in} \sigma(d) = (-1) \sigma(a) \sigma(b)\]
where $m_i$ is a unique Morse critical point of index $i$.  Since $tb(W_a^1 \times W_b^1) = \chi(T^*T^2) = \chi(T^*S^1) = 0$ the only potential nonzero term in theorem \ref{thrm:TB} is the $\tau$-term when computing $tb$ for the triple product. 
Note that for the A chords, $\tau(a \otimes m_0, e) = \tau(a \otimes m_1,e)$, so that 
\[\tau(a \otimes m_0, e)\sigma(a \otimes m_0) \sigma(e) + \tau(a \otimes m_1, e)\sigma(a \otimes m_1) \sigma(e) = 0\]
since the signs of the Morse critical points $m_0,m_1$ cancel.  Similarly, the $\tau$ terms involving the B chords cancel.  Now, 
\[\tau(c,e)\sigma(c)\sigma(e) = \left\{ \begin{array}{l l}
    1 & \quad \text{if $|a - b| < e$}\\
    -1 & \quad \text{if $|a - b| > e$}
  \end{array} \right. \]
\[\tau(d,e)\sigma(d)\sigma(e) = \left\{ \begin{array}{l l}
    -1 & \quad \text{if $a + b < e$}\\
    1 & \quad \text{if $a + b > e$}
  \end{array} \right. \]
since $\tau(a,b) = -1$ and $\sigma(a) = \sigma(b) = \sigma(c) = -1$.  
Thus,
\[tb(W_a^1 \times W_b^1 \times W_e^1) = \left\{ \begin{array}{l l}
		2 & \quad \text{if $a + b > e$ and $|a-b| < e$} \\
    0 & \quad \text{otherwise}
  \end{array} \right. \]
In other words, $tb = 2$ if $(a,b,c)$ satisfy the triangle inequality and $tb = 0$ otherwise.
\end{example}

\subsection{Knots}
The product of two Legendrian knots is a torus, whose Euler characteristic and Thurston-Bennequin invariant are 0.  However, there are interesting $tb$ calculations for products of three Legendrian knots.

\begin{example}
Let $K_1, K_2,K_3$ be a collection of three Legendrian knots with Reeb chords $\{ a_i\}, \{ b_j\}, \{ c_k\}$.  Define
\[\tau(a,b,c)= \left\{ 
\begin{array}{l l}
		2 & \quad \textrm{if $a + b > c$ and $|a-b| < c$} \\
    0 & \quad \textrm{otherwise}
  \end{array} \right. \]
Then the arguments in example 5.2 can be repeated for each triple of Reeb chords $(a,b,c)$ and the Thurston-Bennequin invariant is
\[tb(K_1 \times K_2 \times K_3) = \sum_{i,j,k} \tau(a_i,b_j,c_k)\sigma(a_i)\sigma(b_j)\sigma(c_k)\]
Take $K_1$ to be a once-stabilized unknot, $K_2$ a standard unknot after applying a Reidemeister-1 move, and $K_3$ a right-handed trefoil.  These can be chosen to have the front projections and Lagrangian projections depicted in figure 1. 
 
\begin{figure}
\label{fig:projections}
\centering
	\begin{subfigure}[b]{0.3\textwidth}
			\labellist
				\small\hair 2pt
				\pinlabel $a_1$ at 115 85
				\pinlabel $a_2$ at 380 85
			\endlabellist
			\centering
			\includegraphics[width=\textwidth]{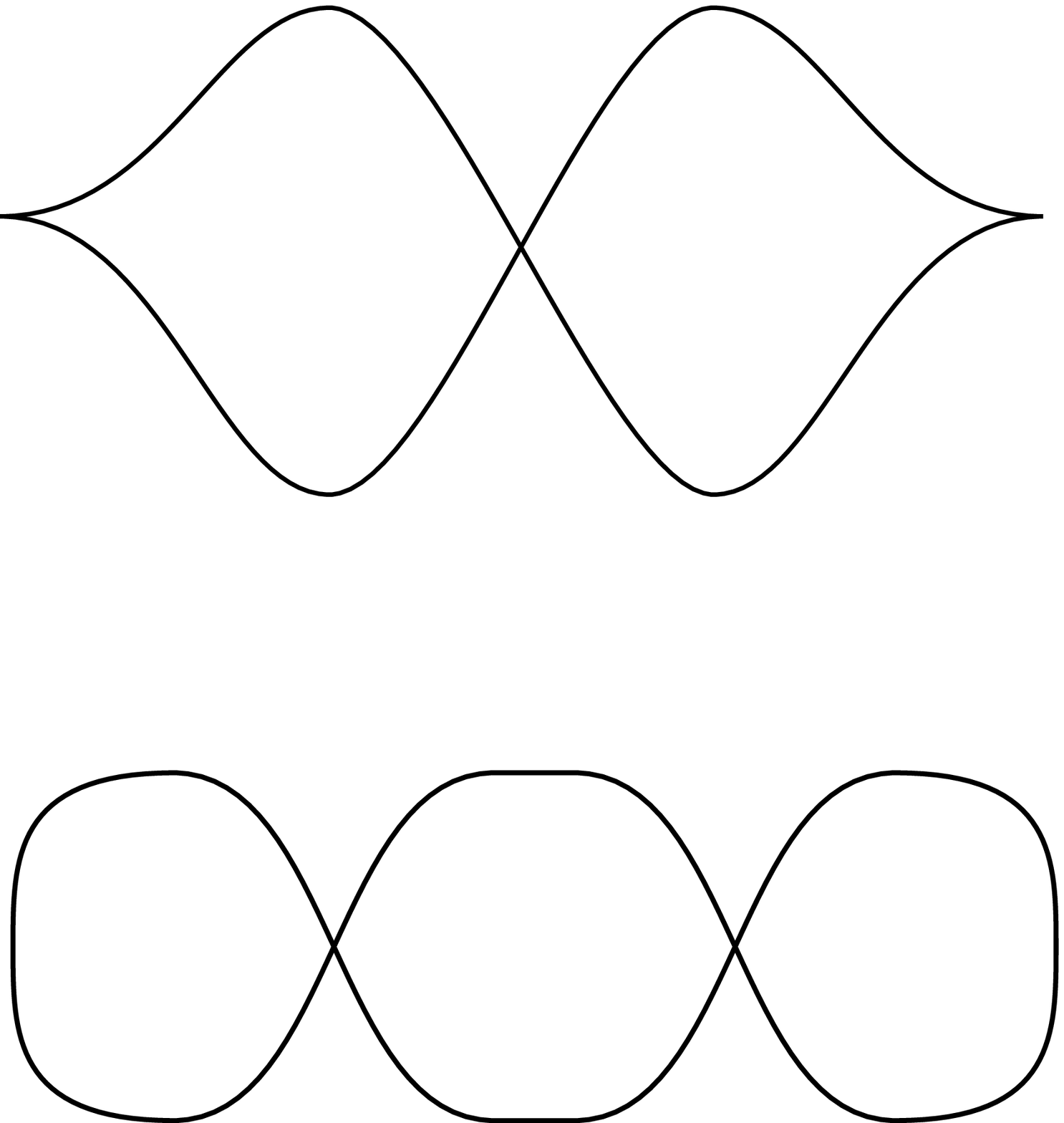}
			\caption{stabilized unknot}
			\label{fig:stabUnknot}
	\end{subfigure}
	\qquad \qquad \qquad
	\begin{subfigure}[b]{0.3\textwidth}
			\labellist
				\small\hair 2pt
				\pinlabel $b_1$ at 175 120
				\pinlabel $b_2$ at 340 170
				\pinlabel $b_3$ at 340 65
			\endlabellist
			\centering
			\includegraphics[width=\textwidth]{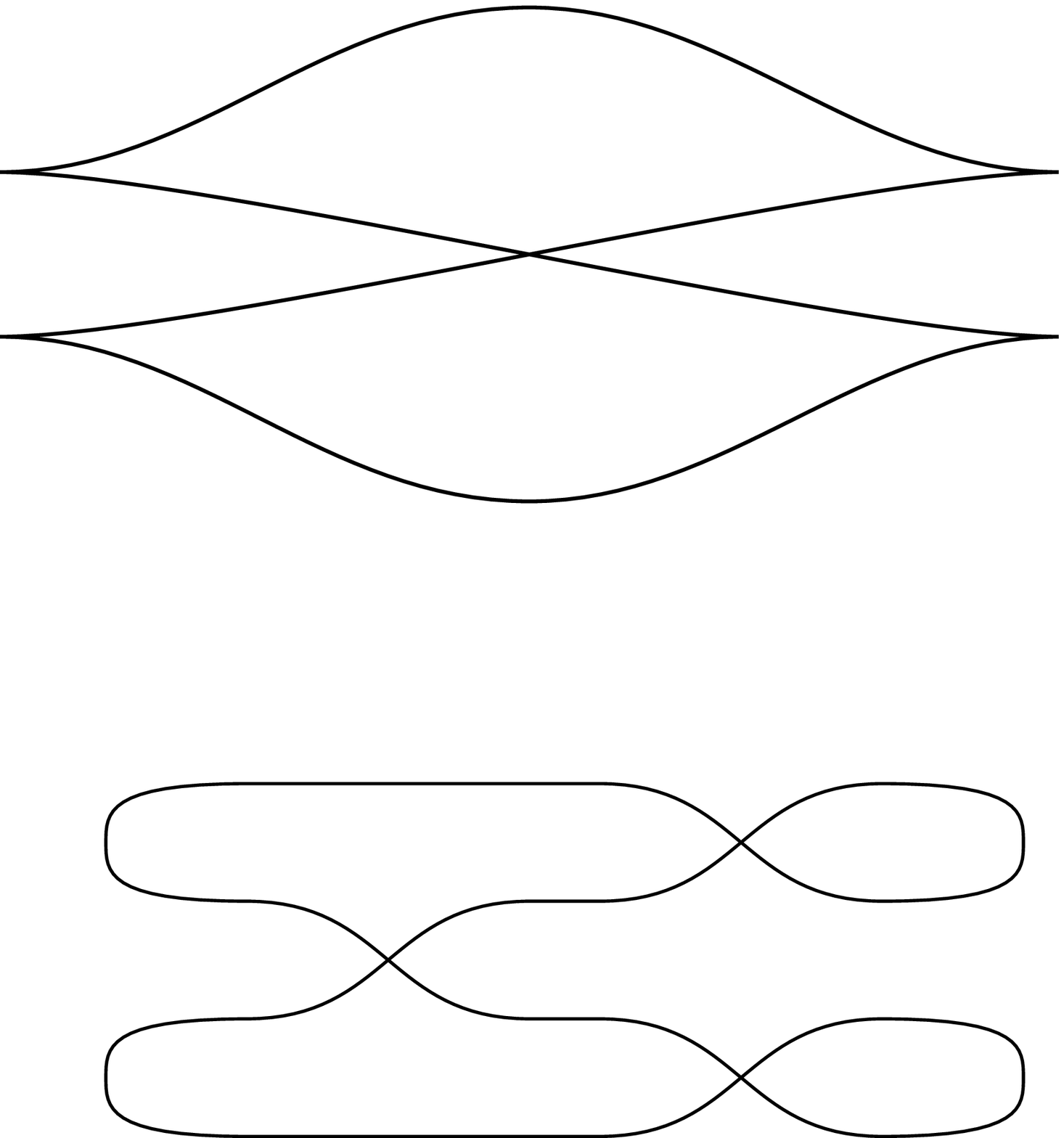}
			\caption{unknot after Reidemeister I}
			\label{fig:reideUnknot}
	\end{subfigure}
	\hspace{.2in}
	\begin{subfigure}[b]{0.3\textwidth}
			\labellist
				\small\hair 2pt
				\pinlabel $c_1$ at 100 120
				\pinlabel $c_2$ at 200 120
				\pinlabel $c_3$ at 300 120
				\pinlabel $c_4$ at 400 180
				\pinlabel $c_5$ at 400 70
			\endlabellist
			\centering
			\includegraphics[width=\textwidth]{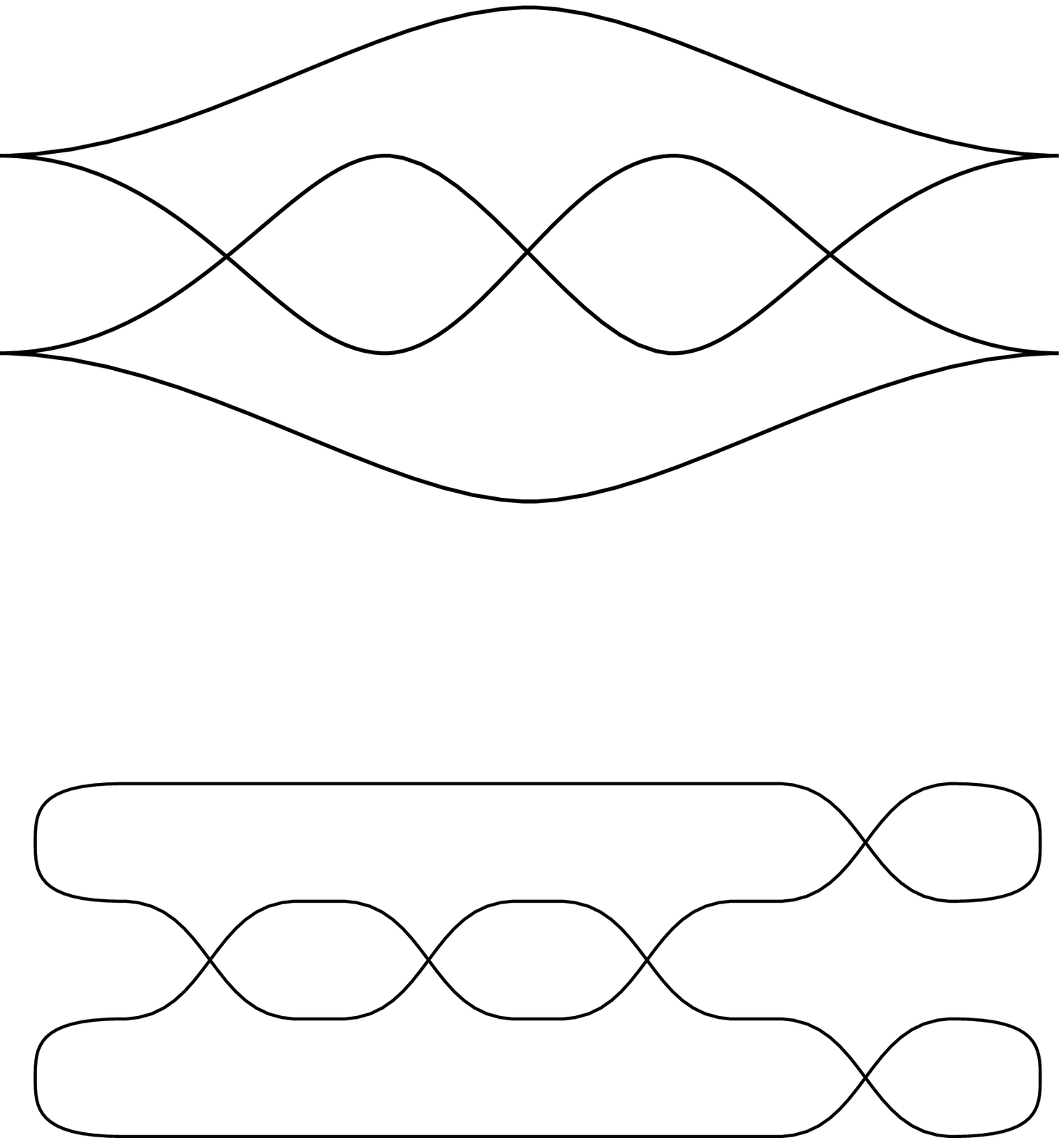}
			\caption{trefoil}
			\label{fig:trefoil}
	\end{subfigure}
\caption{Front and (isotopic) Lagrangian projections of three knots}
\end{figure}

$K_1$ has two Reeb chords $a_1,a_2$; $K_2$ has three Reeb chords $b_1,b_2,b_3$; and $K_3$ has five Reeb chords $c_1,c_2,c_3,c_4,c_5$.  By abuse of notation, $a_i,b_j,c_k$ will refer to both the chord and its action.  These chords have signs
\begin{align*}
\sigma(a_1) &= \sigma(a_2) = -1 \\
\sigma(b_1) &= \sigma(b_2) = -1 & &\sigma(b_3) = 1 \\
\sigma(c_1) &= \sigma(c_2) = -1 & \sigma(c_3) = \sigma(c_4) = &\sigma(c_5) = 1
\end{align*}
so $tb(K_1) = -2$, $tb(K_2) = -1$ and $tb(K_3) = 1$.  The Reeb chord actions cannot be chosen completely arbitrarily as each face of the knot diagram in the Lagrangian projection must satisfy the area identity determined by Stokes's Theorem. For a face $F$ of $K$, the boundary $\partial F$ lies in $\overline{K}$ and the double points of the knot projection split the boundary into segments $\{\overline{\gamma_l}\}$, indexed counter-clockwise.  A corner of the face is \textit{positive} if near the Reeb chord, the $z$-coordinate of $\gamma_{l+1}$ is greater than the $z$-coordinate of $\gamma_{l}$ and \textit{negative} otherwise.
\begin{lemma}
(Area Identity)  Let $A$ be a face of the knot diagram in the Lagrangian projection and $\gamma = \partial A$ be its boundary.  Then
\[\int_A \omega = \sum_p \mathcal{Z}(p) - \sum_n \mathcal{Z}(n)\]
where $p,n$ are the Reeb chords corresponding to the positive and negative corners of $A$
\end{lemma}
Since each face of a knot diagram must have positive area, it follows that 
\begin{align*}
b_1,b_2 &> b_3 \\
c_1,c_2 &> c_3,c_4,c_5
\end{align*}
and that this is the only restriction on the actions.  With this in mind and without loss of generality, let $a,b_+,b_-,c_+,c_-$ refer to some chord in the sets $\{a_1,a_2\},\{b_1,b_2\},\{ b_3\}, \{c_1, c_2\}, \{c_3, c_4, c_5\}$, respectively.

The minimum $tb$ is achieved if all triples $(a,b_+,c_+), (a,b_-,c_-)$ satisfy the triangle inequality and none of the triples $(a,b_+,c_-), (a,b_-,c_+)$ satisfy it, since $\sigma(a)\sigma(b_+) \sigma(c_+) = \sigma(a)\sigma(b_-) \sigma(c_-) = -1$ and $\sigma(a)\sigma(b_+) \sigma(c_-) = \sigma(a)\sigma(b_-) \sigma(c_+) = 1$.  Then
\[tb = (8)(2)(-1) + (6)(2)(-1) +  12(2)(0) + 4(2)(0) = -28\]
For instance, set $a = 5$, $b_+ = c_+ = 10$ and $b_- = c_- = 3$.  

Now, for a 5-tuple $(a,b_+,b_-,c_+,c_-)$, if both triples $(a,b_+,c_-), (a,b_-,c_+)$ satisfy the triangle inequality, then $(a,b_+,c_+)$ must as well since $b_+ > b_-$ and $c_+ > c_-$.  However, this is not true if only one such triple does.  Thus, since there are more $c_-$ chords than $b_-$ chords, the maximum $tb$ is achieved if all possible triples $(a,b_+, c_-)$ satifies the triangle inequality but no other triple does.  For instance, set $a = 5$, $b_+ = 6$, $c_+ = 12$ and $b_- = c_- = 2$.  Then
\[tb = (8)(2)(0) + (6)(2)(0) + 12(2)(1) + 4(2)(0) = 24\]
By switching the values assigned to the $b$'s and $c$'s, it's possible to achieve many $tb$ values between $-28$ and $24$. 
\end{example}

\end{document}